\documentclass[12pt]{article}

\usepackage{amsmath}
\usepackage{amsfonts}
\usepackage{amssymb}
\usepackage{amsthm}

\usepackage{natbib}
\usepackage{url}

\newcommand{\E}{\mathbb{E}}
\newcommand{\Var}{\mathrm{Var}}
\newcommand{\Cov}{\mathrm{Cov}}

\newtheorem{proposition}{Proposition}
\newtheorem{corollary}{Corollary}
\newtheorem{remark}{Remark}
\newtheorem{example}{Example}
\newtheorem{definition}{Definition}
\newtheorem{lemma}{Lemma}
\newtheorem{theorem}{Theorem}

\begin{document}

\title{Covariance as a commutator }
\author{
  Carlos García Meixide\thanks{Work done while at the University of California, Berkeley.} \\
  Instituto de Ciencias Matemáticas, CSIC \\
  Departamento de Matemáticas, Universidad Autónoma de Madrid \\
  \texttt{carlos.garcia@icmat.es}
}
\maketitle

\begin{abstract}
The covariance between real finite variance random variables can be expressed as the commutator of taking expectations and multiplying, both viewed as operators extended to act jointly on pairs of functions. The efficient influence curve of the mean represents a centering operator which we demonstrate to interact with expectations and products through simple commutator identities. These expressions reveal an underlying Lie algebraic structure that endows the calculus of efficient influence curves with a natural differential geometric interpretation. 
\end{abstract}

\section{Introduction}

Let $X,Y\in L^2(P)$ for a distribution $P$ on a measurable space. The covariance $\Cov(X,Y)= \E[X\cdot Y]- \E[X] \cdot  \E[Y]$ measures the failure of multiplicativity of expectation. This observation admits a precise operator form: the covariance is the commutator of the expectation functional with the pointwise product on $L^2(P)$, loosely speaking 
$$\Cov= (\mathbb E \circ \cdot) - (\cdot \circ \mathbb E) $$

We formalize this in a Lie algebra of statistical operators on  $L^2(P)$ and embed it into semiparametric theory, where the efficient influence curve (EIC) of the mean corresponds to the centering operator $T:f\mapsto f- \E[f]$.

The EIC, also known as the canonical gradient, is a central object in contemporary statistics. Originating in the robust statistics literature \citep{huber1964robust,hampel1974influence}, it provides the most sensitive linear approximation to a statistical functional under the unrestricted model, and plays a key role in efficiency theory \citep{bickel1993efficient}.  
In modern causal inference, EICs are essential pieces for targeted learning \citep{vanderlaan2006targeted} and debiased machine learning \citep{chernozhukov2018double}, enabling the construction of estimators that achieve the semiparametric efficiency bound even when nuisance functions are estimated with flexible machine learning, data adaptive methods.

Finding EICs is often nontrivial \citep{kennedy2023, Hines_2022}.  
Nevertheless, once the EIC of an estimand is known, it directly yields efficient estimators satisfying the nonparametric efficiency bound.  
This bound generalises the Cramér–Rao lower bound to infinite-dimensional models and quantifies the minimal possible asymptotic variance among all regular estimators. To illustrate, consider the mean $\psi(P) = {E}_P[Y]$ of an i.i.d. sample $Y_1, \dots, Y_n$.  
Its EIC is simply $\varphi(Y;P) = Y - \psi(P)$, with variance $\Var_P(\varphi)/n=\Var_P(Y)/n$, yielding the sample mean $\bar{Y}_n$ as the efficient estimator.  
The classical central limit theorem gives
\[
\sqrt{n}\{\bar{Y}_n - \psi(P)\} \ \overset{d}{\to} \ \mathcal{N}\!\left(0, \, \Var_P(Y)\right),
\]
and the nonparametric Cramér–Rao bound asserts that no regular estimator can have asymptotic variance smaller than $\Var_P(Y)$.  
For more complex functionals—such as average treatment effects, conditional effect measures, or mediation parameters—the EIC encapsulates, in a single expression, both the form of an optimal estimator and the efficiency bound it attains.

The EIC behaves mathematically like a derivative of the parameter map along regular submodels. As such, it inherits the standard properties of a gradient: additivity (linearity in sums of parameters), the Leibniz rule (product rule for parameters that factor), and the chain rule (composition of parameters through smooth functions). These identities, familiar from calculus, become powerful algebraic tools in semiparametric theory: once the EICs of simpler building blocks are known, the EIC of more complex functionals can be assembled without recomputing from first principles \citep{schuler2025introduction}. This “gradient algebra” viewpoint clarifies why EICs are so popular in modern causal inference, where estimands are often constructed compositionally from elementary components. In our framework, these differential identities acquire a deeper algebraic interpretation: we show that the EIC calculus is governed by commutator relations between expectation and multiplication operators on $L^2(P)$. This connects the differential structure of semiparametric theory to an operator algebraic setting. We study the commutators that are needed to establish basic identities and prove a Jacobi identity involving these brackets.

\section{An extended statistical algebra}
We briefly recall notions from \citet{bickel1993efficient}. Let $\mathcal M$ be a dominated model on a measurable space $(\mathcal Z,\mathcal A)$ and $P\in\mathcal M$. A regular parametric submodel $\{P_\varepsilon:\varepsilon\in(-\eta,\eta)\}\subset\mathcal M$ through $P$ has score $s=\left.\partial_\varepsilon \log p_\varepsilon\right|_{\varepsilon=0}\in L_0^2(P)$, the zero mean subspace of $L^2(P)$. The \emph{tangent space} $T(P)$ is the closure in $L_0^2(P)$ of all such scores. A functional $\psi:\mathcal M\to\mathbb R$ is \emph{pathwise differentiable} at $P$ if there exists $\phi_\psi\in L_0^2(P)$ such that, for every regular submodel with score $s\in T(P)$,
\[
\left.\frac{d}{d\varepsilon}\psi(P_\varepsilon)\right|_{\varepsilon=0}= \E[\phi_\psi\,s].
\]
The projection of $\phi_\psi$ onto $T(P)$ is the \emph{canonical gradient} or \emph{efficient influence curve}. For the mean functional $\psi(P)= \E[U]$, the EIC is $\phi_\psi=U- \E[U]$ in the nonparametric model, which induces the centering operator 
\[
\begin{aligned}
	T : &L^2(P) \longrightarrow L^2(P), \\
&	f \longmapsto f -  \E[f].
\end{aligned}
\]
 We focus on the case where the parameter has the form $\psi(P) = {E}_P[U] = PU$ for some $U \in L^2(P)$, i.e., the population mean. Let $H=L^2(P)$ be the Hilbert space of square-integrable random variables on $(\mathcal{Z},\mathcal{A},P)$ with inner product 
$\langle f,g \rangle = \mathbb{E}_P[fg]$. Denote by $\otimes:H\times H\to H$ the pointwise product $(f,g)\mapsto fg$. Let $P:H\to\mathbb R$ be the expectation operator $Pf= \E[f]$, and let $\eta:\mathbb R\to H$ embed scalars via $\eta(a)=a\cdot 1$. Define the centering operator $T=\mathrm{Id}_H-\eta\circ P$, so $Tf=f-Pf$. 

\begin{proposition} The following decomposition holds
	\label{prop:decomposition}
	\[
	H = \eta(\mathbb{R}) \oplus \ker P,
	\]
	where $\eta(\mathbb{R}) = \{c \cdot 1: c \in \mathbb{R}\}$ is the one-dimensional space of constant functions, and $\ker P = L^2_0(P) = \{f \in L^2(P): {E}_P[f] = 0\}$ is the mean-zero subspace.
\end{proposition}

\begin{proof}
	The Riesz representation theorem ensures that there exists a unique $h_0 \in H$ such that $Pf = \langle f, h_0 \rangle$ for all $f \in H$.
	Since $\langle f, 1 \rangle = \mathbb{E}[f]$, we have $h_0 \equiv 1$, and thus $P$ is the orthogonal projection onto $\mathrm{span}\{1\} = \eta(\mathbb{R})$. The kernel of $P$ is $\ker P = \{f \in H: \mathbb{E}[f] = 0\} = L^2_0(P)$,
	and for any $c \in \mathbb{R}$ and $g \in L^2_0(P)$, we have $\langle c\cdot 1, g \rangle = c\,\mathbb{E}[g] = 0$.
	Hence $\eta(\mathbb{R}) \perp L^2_0(P)$. Finally, for any $f \in H$, set $c = \mathbb{E}[f]$ and $g = f - c \cdot 1$. Then $g \in L^2_0(P)$ and $f = c \cdot 1 + g$.
	Thus $H = \eta(\mathbb{R}) \oplus \ker P$.
\end{proof}

\begin{remark}[Parameters as quotient classes]
	Every $X\in H$ decomposes uniquely as
	\[
	X = \eta(PX) + X - \eta(PX), \qquad X - \eta(PX)\in \ker P.
	\]
	We say two random variables represent the same parameter value if they differ by an element of $\ker P$:
	\[
	X \sim Y \quad \iff \quad X-Y \in \ker P.
	\]
	The space of parameters can thus be identified with the quotient $H/\ker P$, which is isomorphic to $\eta(\mathbb{R})$. In this quotient, $X$ and $\eta(PX)$ are literally the same element. We define $T$ on equivalence classes by
	\[
	T([X]) := X - PX,
	\]
	which is well-defined because $X$ and $\eta(PX)$ belong to the same class $[X]$ and yield the same $X-PX \in \ker P$.  
	In particular,
	\[
	T(X) = T(PX)
	\]
	is not an abuse of notation: it simply reflects that both $X$ and $PX$ are two representatives of the same class in $H/\ker P$.
\end{remark}

\begin{example}[Variance]
	Let $X \in L^2(P)$ with mean $\mu_X = PX$. The variance can be written as
	\[
	\sigma_X^2 = P(X^2) - \mu_X^2.
	\]
	We have
	\[
	T\big(P(X^2)\big) = X^2 - P(X^2), 
	\quad
	T(\mu_X^2) = 2\mu_X\,T(\mu_X) = 2\mu_X\,(X - \mu_X).
	\]
	Therefore,
	\[
	T(\sigma_X^2) 
	= \big(X^2 - P(X^2)\big) - 2\mu_X\,(X - \mu_X)
	= (X - \mu_X)^2 - P\big((X - \mu_X)^2\big).
	\]
	The quotient space perspective guarantees that whether we write $T\left((X - \mu_X)^2\right)$ or $TP\left((X - \mu_X)^2\right)$ , we mean the same EIC.
\end{example}

In the following, we use two forms of slotwise composition. If $M:H\times H\to V$, with $V$ being $H$ or $\mathbb{R}$, is bilinear and $A_1,A_2:H\to H$ are linear, set
\[
M\circ(A_1,A_2):(x,y)\mapsto M(A_1x,A_2y).
\]
If $\ell:H\to\mathbb R$ is linear, define
\[
(\ell\circ M)(x,y)=\ell(M(x,y)),\qquad M\circ(\ell,\ell)(x,y)=M(\eta\,\ell x,\eta\,\ell y).
\]
We overload $\otimes$ for scalar multiplication on $\mathbb R$, writing $\otimes_{\mathbb R}$ when needed.

\begin{definition}[Commutators]
\label{def:comm}
We define the following brackets:
\begin{align*}
[P,\otimes]\ &:=\ (P\circ\otimes)\ -\ \otimes_{\mathbb R}\circ(P,P)\quad:\ H\times H\to\mathbb R,\\
[\otimes,T]\ &:=\ \otimes\circ(T,T)\ -\ T\circ\otimes\quad:\ H\times H\to H,\\
[T,P]\ &:=\ (x,y)\mapsto \big(T(Px),\,T(Py)\big)\quad:\ H\times H\to H\times H.
\end{align*}
\end{definition}

In the last definition we omit the term $P \circ T$ since EICs are mean-zero, making $PT = 0$ on their domain. These are well-typed maps: the first measures the defect of multiplicativity of $P$; the second compares ``multiply then center'' to ``center then multiply''; the third applies centering to the coordinatewise expectations.

\begin{proposition}[Covariance as a commutator]
\label{prop:cov}
For all $X,Y\in L^2(P)$,
\[
[P,\otimes](X,Y)=P(XY)-(PX)(PY)=\Cov_P(X,Y).
\]
\end{proposition}

\begin{proof}
By Definition~\ref{def:comm}, $(P\circ\otimes)(X,Y)=P(XY)$ and $\otimes_{\mathbb R}(PX,PY)=(PX)(PY)$.
\end{proof}

\vspace{0.5cm}

\begin{proposition}[Product--centering commutator]
\label{prop:prod-cent}
For all $X,Y\in L^2(P)$,
\[
[\otimes,T](X,Y)=(TX)(TY)-T(XY)=(X-\mu_X)(Y-\mu_Y)-\big(XY-P(XY)\big).
\]
\end{proposition}

\begin{proof}
Immediate from the definition of $T$ and bilinearity of $\otimes$. Recall the identity $T(U)=T(PU)$ ensuring the  calculations are unambiguous whether starting from the variable or from its mean. 
\end{proof}
\vspace{0.5cm}

\begin{proposition}[Centering--expectation commutator]
\label{prop:cent-exp}
For all $X,Y\in L^2(P)$,
\[
[T,P](X,Y)=\big(T(PX),\,T(PY)\big)=(X-\mu_X,\,Y-\mu_Y).
\]
\end{proposition}

\begin{proof}
By definition $[T,P](X,Y)$ records the centered coordinates $(TX,TY)$, equal to $(X-\mu_X,Y-\mu_Y)$.
\end{proof}
\vspace{0.5cm}

The next two results look almost like commutator relations

\begin{corollary}
$T(PX)T(PY) + T(PXPY) = T(PXY) + \operatorname{Cov}_P(X,Y)$
\end{corollary}

\begin{proof}
$T(PX)T(PY) = (X - PX)(Y-PY) = XY - XPY - YPX + PXPY$. By Leibnitz's rule,
$T(PXPY) = (X- PX)PY + PX(Y-PY)$. Adding both terms, everything cancels out but $XY - PXPY = XY - P(XY) + P(XY) - PXPY = T(PXY) + \operatorname{Cov}_P(X,Y). $
\end{proof}

\vspace{0.5cm}
\begin{corollary}\label{cor:cov}
$T(PX)T(PY) = T(\operatorname{Cov}_P(X,Y)) + \operatorname{Cov}_P(X,Y)$
\end{corollary}

\begin{proof}
Immediate using linearity of differentiation: $T(\operatorname{Cov}_P(X,Y))= T((PXY) - PXPY) = T(PXY) - T(PXPY)$. 
\end{proof}

\section{A Jacobi identity}

We now establish a Jacobi identity for the triple $(P,\otimes,T)$.

\begin{lemma}
\label{lem:pieces}
For $X,Y\in L^2(P)$,
\begin{align*}
[T,[P,\otimes]](X,Y)\ &=\ T\big(\Cov_P(X,Y)\big)-\Cov_P(TX,TY)
= T\big(\Cov_P(X,Y)\big)-\Cov_P(X,Y),\\
[P,[\otimes,T]](X,Y)\ &=\ \Cov_P(X,Y)- (TX)(TY)+T(PX\,PY),\\
[\otimes,[T,P]](X,Y)\ &=\ (TX)(TY)-T(PXY).
\end{align*}
\end{lemma}

\begin{theorem}[Jacobi identity]
\label{thm:jacobi}
For all $X,Y\in L^2(P)$,
\[
\boxed{[T,[P,\otimes]]+[P,[\otimes,T]]+[\otimes,[T,P]]=0}
\]
\end{theorem}

\section{Discussion}
We have revealed an unexpected connection between foundational concepts in statistics—such as variance and covariance, studied since the very origins of the discipline—and modern mathematical structures used to integrate machine learning into causal inference. By expressing variance as a commutator of expectation and multiplication operators, and by situating this identity within the calculus of EICs, we connect the geometry of statistical functionals to the operator theoretic tools increasingly central to contemporary methods like targeted learning and double machine learning.

Our results complement and extend the established principles of additivity, the Leibniz rule, and the chain rule that are popularly used in the derivation of EICs \citep{schuler2025introduction}. These rules form gradient algebra—the framework that allows complex EICs to be constructed from simpler components in the same way that complex derivatives are built from elementary ones. While the literature has long exploited these calculus analogues for theoretical derivations, their role has often been treated as implicit rather than a general-purpose toolkit.

An immediate question arising from our results is whether the commutator identities we established in the case of the mean extend to more general estimands. While the mean offers a particularly simple structure, other parameters such as stochastic intervention effects or functionals defined via estimating equations may interact in more intricate ways. For such estimands, it is not obvious whether the same algebraic relations persist. Characterizing these generalized commutator formulas could reveal deeper operator-theoretic structures underlying EICs and provide a unifying language for efficiency theory beyond the mean.

The algebraic structure underlying statistical functionals has immediate practical implications for the design of causal inference strategies. This perspective aligns naturally with recent developments that automates the construction of efficient estimators without requiring manual EIC derivations \citep{luedtke2025}. A principled understanding of the operator structure ensures that such primitives can be implemented modularly and reused across problems, making semiparametric efficiency not only theoretically transparent but also computationally accessible to a much broader audience.

Geometrically understanding statistical functionals is crucial for the future of causal inference, which relies nowadays heavily on EICs and other constructions based on operator theory that are, at their core, manifestations of deeper identities and functional analysis decompositions. Revealing and exploiting this structure not only clarifies why existing methods work but also opens systematic pathways to discovering new estimators, deriving efficiency bounds, and eventually designing algorithms that are valid under weaker assumptions. As causal inference continues to expand into complex longitudinal settings, such foundational insights will be indispensable for ensuring both theoretical rigor and practical reliability.

\bibliographystyle{apalike}
%\bibliography{cova}

\end{document}